\documentclass[[amscd,amssymb,verbatim,12pt]{amsart}
\usepackage{amssymb, latexsym, amsmath, amscd, array, graphicx}
\usepackage{hyperref}
\usepackage[all]{xy}
\usepackage{textcomp}
\usepackage{mathrsfs}
\swapnumbers \numberwithin{equation}{section}

\theoremstyle{plain}

\newtheorem{thm}{Theorem}[section]
\newtheorem{theorem}[thm]{Theorem}

\newtheorem{lemma}[thm]{Lemma}

\newtheorem{prop}[thm]{Proposition}
\newtheorem{cor}[thm]{Corollary}

\theoremstyle{definition}
\newtheorem{defn}[thm]{Definition}

\newtheorem{question}[thm]{Question}

 \newcommand{\Wi}{\widetilde}

\def\Z{{\mathbb Z}}

\def\1{\hbox{\rm\rlap {1}\hskip.03in{\rom I}}}
\def\Bbbone{{\rm1\mathchoice{\kern-0.25em}{\kern-0.25em}
{\kern-0.2em}{\kern-0.2em}I}}

\long\def\forget#1\forgotten{} %

\newcommand\ver[1]{\marginpar{\tiny Changed in Ver \VER}}

\newcommand{\mc}{ \text {mc}}

\date{\today}

\begin{document}

\title[On Gromov's conjecture]{Gromov's Conjecture for Product of Baumslag-Solitar groups and some other One-relator groups}


\author[S.~Howladar]{Satyanath Howladar}



\address{Satyanath Howladar, Department of Mathematics, University
of Florida, 358 Little Hall, Gainesville, FL 32611-8105, USA}
\email{showladar@ufl.edu}



\begin{abstract}
We show that Baumslag-Solitar groups are virtually 2-avoidable, that is, they admit finite index subgroup whose first homology is devoid of $\Z_2$ summand. We also prove virtual 2-avoidability for some other classes of one-relator groups, which generalizes non-orientable surface groups. This along with result from a previous paper confirms Gromov's Conjecture about macroscopic dimension of universal cover of PSC manifolds, for all closed oriented spin manifolds whose fundamental group is product of Baumslag-Solitar groups, the one-relator groups under consideration.
\end{abstract}



\maketitle
\section{Introduction}

The notion of macroscopic dimension was introduced by M. Gromov~\cite{G2} to study topology of manifolds with a positive scalar curvature metric.

\begin{defn}
 A metric space $X$ has the macroscopic dimension $\dim_{\mc} X \leq k$ if
there is a uniformly cobounded proper map $f:X\to K$ to a $k$-dimensional simplicial complex.
Then $\dim_{mc}X=m$ where $m$ is minimal among $k$ with $\dim_{\mc} X \leq k$.
\end{defn}
\smallskip
A map of a metric space $f:X\to Y$ is uniformly cobounded if there is a uniform upper bound on the diameter of preimages $f^{-1}(y)$, $y\in Y$.

\

{\bf Gromov's Conjecture.} {\it The macroscopic dimension of the universal covering $\Wi M$ of  a closed  positive scalar curvature  $n$-manifold $M$ satisfies the inequality $\dim_{mc}\Wi M\leq n-2$ for the metric on $\Wi M$ lifted from $M$}.

\

The above Conjecture depends heavily on the fundamental group $\pi_1(M)$ of the manifold because of the easy fact:

\begin{lemma}
Let $\Gamma'\subset \Gamma$ be a finite index subgroup of $\Gamma$. If Gromov's Conjecture holds for manifolds having fundamental group $\Gamma'$, then it also holds true for manifolds with fundamental group $\Gamma$.
\end{lemma}

\begin{proof}
Let $M$ be a closed manifold with $\pi_1(M)=\Gamma$, and $M'$ be the finite cover of $M$ corresponding to $\Gamma'$. If $M$ admits positive scaler curvature metric, then by lifting this metric to $M'$, we have PSC metric on $M'$. Since their universal cover coincide hence Gromov's conjecture for $M'$ implies the same for $M$.  
\end{proof}

Thus generally we proceed by restricting our attention to various classes of groups, and prove the Conjecture for manifolds with fundamental groups belonging to the class. For every compact metric space $X$, we consider the constant map $f:X\to pt$ and can choose the uniform bound as $diam(X)$, thus implying $dim_{mc}X=0$. As a result Gromov's conjecture trivially holds for manifolds with finite fundamental groups since universal cover of such manifolds are compact. Essentially Gromov's Conejcture is for manifolds with infinite fundamental groups.

\

The first step in proving the conjecture is to show $\dim_{mc}\Wi M\leq n-1$, for closed PSC $n$-manifold $M$. Gromov defined {\em inessential manifolds} $M$ as those for which a classifying map $u:M\to B\Gamma$ of the universal covering $\Wi M$ can be deformed to the $(n-1)$-skeleton $B\Gamma^{(n-1)}$ where $\dim M=n$. Clearly, for an inessential $n$-manifold $M$  we have $\dim_{mc}\Wi M\leq n-1$.

We call an $n$-manifold $M$ {\em strongly inessential} if a classifying map of its universal covering $u:M\to B\Gamma$ can be deformed to the $(n-2)$-skeleton.
Since for strongly inessential $n$-manifolds $\dim_{mc}\Wi M\le n-2$, the following conjecture implies the Gromov's conjecture.

{\bf Strong Gromov's Conjecture.} {\it  A closed  positive scalar curvature  manifold $M$ is strongly inessential.}

\

We refer to ~\cite{BD1,BD2,Dr1,Dr2,Dr3,DH} for recent progress on the Strong Gromov's Conjecture. In this paper we try to answer a question about 2-dimensional groups that was posed in the recent paper~\cite{DH}, in certain classes of one-relator groups. We recall from \cite{DH}, the following definition:

\begin{defn} A 2-dimensional group (geometric dimension) $\Gamma$ is $2-avoiding$ if its abelianization $\Gamma^{ab}$ or in other words $H_1(\Gamma)$ does not contain $\mathbb Z_2$ as a direct summand. 
\end{defn}

The following question was posed in  \cite{DH}:
\begin{question}\label{d}
Does every 2-dimensional group contain a 2-avoiding finite index subgroup?
\end{question}

Question~\ref{d} came up in the endeavor of trying to prove Gromov's Conjecture for product of 2-dimensional groups. We refer to \cite{DH} for more about the Gromov's Conjecture, macroscopic dimension.  In the paper  \cite{DH} we proved:

 \begin{prop}\label{c}
The Strong Gromov's Conjecture holds for spin manifolds whose fundamental groups are finite products $\Gamma=\Gamma_1\times\dots\times\Gamma_k$
of 2-dimensional finitely generated 2-avoiding groups such that each $\Gamma_i$ satisfies the Strong Novikov Conjecture.
\end{prop}

In Proposition~\ref{c} the condition 2-avoidability of the factors $\Gamma_i$ arose as the consequence of a fact about Moore Spaces $M(\Z_k,m)$. Moore spaces are CW complexes, $M(\Z_m,n)=S^n\cup_\phi D^{n+1}$ where $\phi:S^n\to S^n$ is a map of degree $m$. In the proof of Proposition~\ref{c} we used the following result about Moore Spaces:

\begin{prop}\label{msmash}(~\cite{JN} Corollary 6.6)
If  $gcd(k,l)=d$ is odd or $4$ divides $k$, then $M(\Z_k,m)\wedge M(\Z_l,n)$ is homotopy equivalent to $M(\Z_d,m+n)\vee M(\Z_d,m+n+1)$.
\end{prop}

In particular we have, $$M(\Z_2,m)\wedge M(\Z_2,n)\ne M(\Z_2,m+n)\vee M(\Z_2,m+n+1),$$  for positive integers $m,n$. We can see this already for $m=n=1$, as $$M(\Z_2,1)\wedge M(\Z_2,1)=\mathbb{RP}^2 \wedge \mathbb{RP}^2\ne \Sigma \mathbb{RP}^2 \vee \Sigma^2\mathbb{RP}^2=M(\Z_2,2)\vee M(\Z_2,3).$$

In view of removing the 2-avoidability condition from Proposition~\ref{c}, Question~\ref{d} was posed, whose affirmative answer would prove the original Gromov's Conjecture for the products of all 2-dimensional groups. 
In this paper we give positive answer to Question~\ref{d} restricting our attention to certain classes of torsion free One-relator groups. We considered the following one-relator groups:
 \begin{enumerate}
        \item Baumslag Solitar groups, $$B(m,n)=\langle a,t | ta^{m}t^{-1}=a^{n}\rangle,m,n\in \Z$$
        \item Baumslag Strebel groups, $$G_{m,n,k}=\langle a,t | t^ka^{m}t^{-k}=a^{n}\rangle,m,n,k\in \Z$$
        \item Baumslag Gersten groups, $$BG(m,n)=\langle a,t | tat^{-1}a^mta^{-1}t^{-1}=a^{n}\rangle,m,n\in \Z$$ 
        \item Meskin groups, $$\Gamma=\langle s_1,s_2,...,s_n| s_1^{k_1}s_2^{k_2}...s_n^{k_n}\rangle, n\ge 2, k_i\ge 1.$$
    \end{enumerate} 
    
    Let $\mathscr{F}$ be the union of all above families of groups.

\ 
 
 The main result in this paper is 
     \begin{theorem}\label{m}
   All groups in the family $\mathscr{F}$ are admit a index 2 subgroup which is 2-avoidable.
    \end{theorem}

    It is well known that one-relator group is torsion free if and only if the relator is not a proper power of some word in the generators. Also one-relator group with torsion is virtually torsion free. Torsion free One-relator groups are 2-dimensional by the Lyndon-Cockeroft theorem~\cite{L},\cite{C}. They also satisfy Strong Novikov Conjecture, because they have finite asymptotic dimension~\cite{BD}, \cite{Dr3}, \cite{T}, thus satisfying the coarse Baum-Connes conjecture for both $KU$ and $KO$. Baum-Connes conjecture implies Strong Novikov Conjecture. Thus for finite collection of groups $\Gamma_i\in \mathscr{F}$, we can consider finite index subgroups $\Gamma'_i\subset \Gamma_i$, which are 2-avoidable. Since $\Gamma'_i$ correspond to some finite cover $X_i$ of $B\Gamma$, their universal covers coincide, which is contractible as $\Gamma_i$'s are 2-dimensional. Thus $\Gamma'_i$'s are 2-dimensional and we can apply Proposition~\ref{c} to $\Gamma'_i$'s. We get a corollary to Theorem~\ref{m}:
    \begin{cor}
    Gromov Conjecture holds for closed spin manifolds whose fundamental groups are finite products $\Gamma=\Gamma_1\times\dots\times\Gamma_k$, 
for $\Gamma_i\in \mathscr{F}$.
    \end{cor}

\section{Preliminaries}
In this section we collect some basic tools used in the proof of Theorem~\ref{m}. 

\subsection{Smith Normal form of matrix}

 We recall Smith Normal form of a matrix. For $m, n\ge 1$, let $M_{m\times n}(\Z)$ be the set of $m\times n$ matrices with
integer entries

\begin{prop}\label{1}
(The Smith Normal form). Given a nonzero matrix $A\in M_{m\times n}(\Z),$ there exist invertible matrices $P\in M_{m\times m}(\Z)$ and $Q\in M_{n\times n}(\Z)$ such that $$PAQ=\begin{pmatrix}  
 n_1 & 0 & \cdots & 0 & 0 &  \cdots & 0\\         
 0 & n_2 & \cdots & 0 & 0 &\cdots & 0 \\
 \vdots & \vdots & \ddots & \vdots & \vdots & \ddots & \vdots \\
 0 & 0 & \cdots & n_k & 0 &\cdots & 0 \\
 0 & 0 & \cdots & 0 & 0 &\cdots & 0 \\  
 \vdots & \vdots & \ddots & \vdots & \vdots & \ddots & \vdots \\
        0 & 0 & \cdots & 0 & 0 &\cdots & 0 
\end{pmatrix}.
$$ where the integers $n_i\ge 1$ are unique up to sign and satisfy $n_1|n_2|\cdots |n_k.$ Further, one can compute the integers $n_i$ by the recursive formula $n_i=\frac{d_i}{d_{i-1}}$, where $d_i$ is the greatest common divisor of all $i\times i$-minors of the matrix $A$ and $d_0$ is defined to be 1.

\end{prop}

We refer to \cite{HU} for a proof of above [HU, Proposition 2.11, p.339].
\begin{cor}\label{2}
 Given a nonzero matrix $A\in M_{n\times n}(\Z),$ with $det(A)\ne 0$ there exist invertible matrices $P\in M_{n\times n}(\Z)$ and $Q\in M_{n\times n}(\Z)$ such that $$PAQ=\begin{pmatrix}  
 1 & 0 & \cdots & 0 & 0 &  \cdots & 0\\         
 0 & 1 & \cdots & 0 & 0 &\cdots & 0 \\
 \vdots & \vdots & \ddots & \vdots & \vdots & \ddots & \vdots \\
 0 & 0 & \cdots & 1 & 0 &\cdots & 0 \\
 0 & 0 & \cdots & 0 & n_1 &\cdots & 0 \\  
 \vdots & \vdots & \ddots & \vdots & \vdots & \ddots & \vdots \\
        0 & 0 & \cdots & 0 & 0 &\cdots & n_k 
\end{pmatrix}.
$$ where the integers $n_i\ge 2$ are unique up to sign and satisfy $n_1|n_2|\cdots |n_k.$

\end{cor}

\begin{prop}\label{3}
(Invariant Factor Decomposition of Finitely generated Abelian groups). If $G$ is finitely presented abelian group, and if $A\in M_{m\times n}(\Z)$ is the presentation matrix of $G$ then, $G\cong \Z^{l}\times\Z_{n_1}\times\Z_{n_2}\times...\times \Z_{n_k}$, where $n_i$'s are the invariant factors of $A$, $l$=\#(generators in presentation of $G$) $-$ \#(nonzero entries in Smith normal Form of $A$).
\end{prop}

Proof of Proposition~\ref{3} can be found in [DF, Theorem 3, p.158]

\subsection{About One relator groups}

We collect some results about One relator groups. Let $$\Gamma=\langle s_1,s_2,...,s_n| r\rangle$$ where $r\in F(S)$ and $S=\{s_1,s_2,...,s_n\}$ be an one-relator group. We assume that $n>1$ and all the generators appear in $r$, $r$ is cyclically reduced and $r$ is not a proper power. The last condition implies $\Gamma$ is torsion-free. Clearly all groups in $\mathscr{F}$ satisfy the above conditions thus torsion-free.
\begin{lemma}\label{lem}
    If $\Gamma$ is torsion free one-relator group which is not 2-avoiding, then $H_1(\Gamma)$ has only one $\Z_2$ summand.
\end{lemma}

\begin{proof}
    The presentation 2-complex of $\Gamma$ has one skeleton wedge of $S^1$'s corresponding to each $s_i$, with a 2-cell attached according to the word $r$. Thus it's cellular chain complex is given by, $$0 \to \Z \xrightarrow{d_2} \Z^n \xrightarrow{d_1} \Z\to 0.$$ If after abelianizing the relator $r$ we get $r_{ab}=k_1s_1+k_2s_2+...+k_ns_n$, then $d_2(1)=(k_1,k_2,...,k_n)^T$, where we are treating it as column vector. Clearly $d_1=0$. Let $d:=gcd(k_1,k_2,...,k_n)$. Thus by Smith Normal Form reduction of the boundary matrix, which is $(k_1,k_2,...,k_n)^T$ in our case we get, $$H_1(\Gamma)=\Z^n/(k_1,k_2,...,k_n)\Z=\begin{cases}
\Z^{n-1}\oplus\Z_d & \text{if }  d\ge 2,\\
\Z^{n-1}  & \text{if } d=1.
\end{cases} $$

If $\Gamma$ is not 2-avoidable then $d=2$, hence done. Moreover, we observe that all the $k_i$'s must be even, and there is at least one $k_i$ such that $2\mid k_i$ but $4\nmid k_i$.

 \end{proof}

\subsection{Fox Derivative and Homology of finite covers}

In order to compute $H_1(\Gamma')$ where $\Gamma'\subset \Gamma$ is a finite index subgroup of a group as in our case, we recall Fox derivative of finite presented groups and Fox's main theorem in this direction.
If $S=\{s_1,s_2,\cdots, s_n\}$, there is a unique $\Z$-linear map $$\frac{\partial }{\partial s_i}:\Z[F(S)]\to \Z[F(S)]$$ satisfying the condition:
\begin{itemize}
        \item For $s_j\in S$ we have,  $$\frac{\partial s_j}{\partial s_i}=\delta_{ij}.$$
      
       \item For any word $u=u_1u_2...u_k$, where $u_i \in F(S)$, $$\frac{\partial u}{\partial s_i} = \frac{\partial u_1}{\partial s_i} + u_1 \frac{\partial u_2}{\partial s_i} + u_1 u_2 \frac{\partial u_3}{\partial s_i} + \cdots + u_1 u_2 \dots u_{k-1} \frac{\partial u_k}{\partial s_i}$$

    \end{itemize} 
  The above conditions imply $$\frac{\partial s_i^{-1}}{\partial s_i}= -s_i^{-1},$$ which helps us compute Fox derivatives for any word $u\in F(S)$ very easily. We will see examples in next sections. Using Fox derivatives of relations in a presentation of group $\Gamma$, one can compute $H_1(\Gamma')$ of any finite index subgroup $\Gamma'\subset \Gamma$:

\begin{theorem}\label{f} (FOX) Let $p: X'\to X$ be a $q$ sheeted covering of a finite CW complex $X$, $\theta: \pi_1(X)\to S_q$ the associated permutational representation, and $\pi_1(X)=\langle x_1,...,x_n| r_1,...,r_m\rangle$ any presentation of $\pi_1(X)$. Then the $nq \times mq$ matrix of integers $\theta(J)$, where $J=((\frac{\partial r_i}{\partial x_j}))$ is a presentation matrix, over $\Z$, for $H_1(X')\oplus \Z^{q-1}$.
\end{theorem}

The above theorem is a modified version of Fox's original theorem~\cite{F}, which was restated and simplified by John Hempel in \cite{JH} where the proof of Theorem~\ref{f} can be found. In the above theorem $S_q$ denotes the group of permutation of $\{1,2,...,q\}$, or equivalently a subgroup of $GL_q(\Z)$ of matrices which permutes the standard basis of $n$-dimensional vector space. We follow notations and description as in \cite{JH}. For $\sigma\in S_q$ we consider action on right $$i\mapsto i \sigma$$ or the matrix $$P_\sigma=((\delta_{i\sigma,j})).$$ Given $p: X'\to X$ a $q$ sheeted cover, permutational representation $\theta: \pi_1(X)\to S_q$, is defined as follows: Let $H=p_*\pi_1(X',x')$, and let the right cosets of $H$ be $H=H_1,H_2,\cdots, H_q$.  For $g\in \pi_1(X)$, $$i\theta(g)=j \Leftrightarrow H_ig=H_j.$$

It is known that $\theta$ is well defined up to conjugation, $\theta(\pi_1(X))$ is a transitive subgroup of $S_q$, $H = \theta^{-1}\{\sigma \in S_q: 1\sigma = 1\}$, and, in fact, the correspondence $$(p:X'\to X) \leftrightarrow \theta$$
is a one-to-one correspondence between equivalence classes of $q$-sheeted coverings of $X$ with conjugacy classes of representations of $\pi_1(X)$ onto transitive subgroups of $S_q$.

 \section{Baumslag-Solitar Groups} In this section we give a direct proof of Theorem~\ref{m} for Baumslag Solitar groups. First we prove an elementary fact when $B(m,n)$ is non-$2$-avoidable.

 \begin{lemma}\label{bs}
	$B(m,n)$ is non-$2$-avoidable if and only if $|m-n|=2$.
\end{lemma}

\begin{proof}
 We have $$B(m,n)^{ab}=\langle a,t | (m-n)a=0\rangle.$$ Thus by Lemma~\ref{lem},  $H_1(B(m,n))=\Z\oplus \Z_{|m-n|}$. 
 \end{proof} 

 Because of Lemma~\ref{bs} we will be considering only above type of Baumslag-Solitar groups: $B(n+2,n)$, $B(n,n+2)$ for our purpose. Other $B(m,n)$'s with $|m-n|\ne 2$ are already $2$-avoidable. 

 \begin{prop}
	For all $n\in \Z$, $B(n, n+2)$ has an index $2$ subgroup $\Gamma'$ which is 2-avoidable.
\end{prop}

\begin{proof}

Case 1: $n=2p$. Let $\Gamma=B(2p,2p+2)=\langle a,t | ta^{2p}t^{-1}=a^{2p+2}\rangle$. Consider the surjective homomorphism, $$\phi:\Gamma \to \Z_2, \phi(a)=1, \phi(t)=0.$$ Clearly $\Gamma'=ker(\phi)$, which has index 2. Let $X$ be the 2-to-1 cover of $B\Gamma$ corresponding to $\Gamma'$. The 1-skeleton $X^{(1)}$ has: 2 vertices $v_{+}, v_{-}$, joined by two edges $a_{+}, a_{-}$, which are lifts of loops represented by generator $a\in \Gamma$. Let's assume $a_{+}$ starts at $v_{+}$, and end at $v_{-}$, similarly for $a_{-}$. the generator $t\in ker(\phi)=\Gamma'\subset \Gamma$, hence it lifts as loops in $X$, say $t_{+}$ at $v_{+}$ and $t_{-}$ at $v_{-}$. To construct 2-skeleton $X^{(2)}$ we need to attach lifts of 2-cell corresponding to relator $r$. There are 2 lifts of $r$, say $r_{+}, r_{-}$. Then $r_{+}$ is represented by the word 

$$r_{+}=t_{+}(a_{-}a_{+})^{p}(t_{+})^{-1}((a_{+})^{-1}(a_{-})^{-1})^{p+1}$$ 

$$r_{-}=t_{-}(a_{+}a_{-})^{p}(t_{-})^{-1}((a_{-})^{-1}(a_{+})^{-1})^{p+1}.$$ 

Thus the cellular chain complex of $X$ is given by $$0 \to \Z^2 \xrightarrow{d_2} \Z^4 \xrightarrow{d_1} \Z\to 0.$$ The cellular boundary map $d_2$ is given by the matrix 
$$ d_2=\bordermatrix{ & r_{+} & r_{-} \cr
              t_{+} & 0 & 0 \cr
              t_{-} & 0 & 0 \cr
             a_{+} & -1 & -1 \cr
             a_{-} & -1 & -1 \cr} \sim \begin{pmatrix}  
 1 & 0 \\         
 0 & 0 \\
 0 & 0 \\
 0 & 0           
\end{pmatrix}.
  $$ 
              
             We denote Smith reduction by $\sim$. Hence $H_1(\Gamma')=H_1(X)=\Z$, thus $\Gamma'$ is 2-avoidable.

Case 2: $n=2p+1$. $\Gamma=B(2p+1,2p+3)=\langle a,t | ta^{2p+1}t^{-1}=a^{2p+3}\rangle$. Consider the homomorphism, $$\phi:\Gamma \to \Z_2, \phi(a)=0, \phi(t)=1.$$ As in previous case consider $\Gamma'=ker(\phi)$, which has index 2. Let $X$ be the 2-to-1 cover of $B\Gamma$ corresponding to $\Gamma'$. The 1-skeleton $X^{(1)}$ has: 2 vertices $v_{+}, v_{-}$, joined by two edges $t_{+}, t_{-}$, which are lifts of loops represented by generator $t\in \Gamma$. Let's assume $t_{+}$ starts at $v_{+}$, and end at $v_{-}$, similarly for $t_{-}$. Here $a\in ker(\phi)=\Gamma'\subset \Gamma$ lifts as loops $a_{+}$ at $v_{+}$ and $a_{-}$ at $v_{-}$. Then $r_{+}$ is represented by the word 

$$r_{+}=t_{+}(a_{+})^{2p+1}(t_{+})^{-1}(a_{-})^{-(2p+3)}$$ 

$$r_{-}=t_{-}(a_{-})^{2p+1}(t_{-})^{-1}(a_{+})^{-(2p+3)}.$$ Thus the cellular chain complex of $X$ is given by $$0 \to \Z^2 \xrightarrow{d_2} \Z^4 \xrightarrow{d_1} \Z\to 0.$$ The cellular boundary map $d_2$ is given by the matrix 
$$ d_2=\bordermatrix{ & r_{+} & r_{-} \cr
              t_{+} & 0 & 0 \cr
              t_{-} & 0 & 0 \cr
             a_{+} & 2p+1 & -(2p+3) \cr
             a_{-} & -(2p+3) & 2p+1 \cr} \sim \begin{pmatrix}  
 1 & 0 \\         
 0 & 8(p+1) \\
 0 & 0 \\
 0 & 0           
\end{pmatrix}.
  $$ 

Hence $H_1(\Gamma')=H_1(X)=\Z\oplus \Z_{8(m+1)}$, thus again $\Gamma'$ is 2-avoidable. 
\end{proof}
  For the cases $B(n+2,n)$, exactly similar argument as in the Proposition~\ref{bs} works out.

\section{Fox derivative method and Main Theorem}
In this section we reprove Proposition~\ref{bs} using Fox derivative, and applying Theorem~\ref{f}. This technique extends to other classes of groups in $\mathscr{F}$, thus proving the Main Theorem~\ref{m}.

The general scheme of proving Theorem~\ref{m} is as follows:
 \begin{enumerate}
        \item In each class of groups $\Gamma\in \mathscr{F}$, we find the jacobian, $$J=((\frac{\partial r_i}{\partial x_j})),$$ from it's presentation. Then we restrict our attention to those $\Gamma$'s which are non-$2$-avoidable.
        \item We take a specific surjective homomorphism $\theta:\Gamma\to \Z_2=S_2$, and consider $\Gamma'=ker \theta$, which is clearly a index $2$ subgroup of $\Gamma$.
        \item We compute $\theta (J)$, considering $S_2$ as the permutation matrix group of order 2, generated by $\begin{pmatrix}  
 0 & 1 \\         
 1 & 0 
        
\end{pmatrix}$. Thus $\theta (J)$ is an integers matrix, whose Smith Normal form, gives us $H_1(\Gamma')$ by Proposition~\ref{1},~\ref{3} and Theorem~\ref{f}.
    \end{enumerate}

\subsection{Baumslag-Strebel Groups} We consider a generalised version of $B(m,n)$, known as Baumslag-Strebel Groups $G_{m,n,k}$, considered in \cite{BS}. They have presentation $$G_{m,n,k}=\langle a,t | t^ka^{m}t^{-k}=a^{n}\rangle.$$ Clearly for $k=1$ we get $B(m,n)$ groups. Let $r=t^ka^mt^{-k}a^{-n}$ denote it's relator. We assume $gcd(mn,k)=1$, this ensures $r$ is not proper power. Using the rules for finding Fox derivative we have 

$$\frac{\partial r}{\partial a}=t^k(1+a+...+a^{m-1})-(1+a+...+a^{n-1}),$$ 

$$ \frac{\partial r}{\partial t}=(1-a^n)(1+t+...+t^{k-1}).$$

It is easy to see by abelianisation, $G_{m,n,k}$ is non-$2$-avoidable for $|m-n|=2$. As in proof of Proposition~\ref{bs} we consider only $G_{n,n+2,k}$. 

Case 1: $n=2p, k=2l$. Consider $\theta:\Gamma \to S_2$, $\theta(a)=\begin{pmatrix}  
 0 & 1 \\         
 1 & 0 
        
\end{pmatrix}, \theta(t)=\begin{pmatrix}  
 1 & 0 \\         
 0 & 1 
        
\end{pmatrix}$, we have $$\theta ( \frac{\partial r}{\partial t})=\begin{pmatrix}  
 0 & 0 \\         
 0 & 0 
        
\end{pmatrix}, \theta ( \frac{\partial r}{\partial a})=\begin{pmatrix}  
 1 & 1 \\         
 1 & 1
        
\end{pmatrix},$$ thus $H_1(ker \theta)=\Z$.

Case 2: $n=2p, k=2l-1$. Consider $\theta:\Gamma \to S_2$, $\theta(a)=\begin{pmatrix}  
 0 & 1 \\         
 1 & 0 
        
\end{pmatrix}$, $\theta(t)=\begin{pmatrix}  
 1 & 0 \\         
 0 & 1 
        
\end{pmatrix}$, we have $$\theta ( \frac{\partial r}{\partial t})=\begin{pmatrix}  
 0 & 0 \\         
 0 & 0 
        
\end{pmatrix}, \theta ( \frac{\partial r}{\partial a})=\begin{pmatrix}  
 1 & 1 \\         
 1 & 1
        
\end{pmatrix},$$ thus $H_1(ker \theta)=\Z$.

Case 3: $n=2p-1, k=2l$. Consider $\theta:\Gamma \to S_2$, $\theta(a)=\begin{pmatrix}  
 0 & 1 \\         
 1 & 0 
        
\end{pmatrix}$, $\theta(t)=\begin{pmatrix}  
 1 & 0 \\         
 0 & 1 
        
\end{pmatrix}$, we have $$\theta ( \frac{\partial r}{\partial t})=\begin{pmatrix}  
 0 & 0 \\         
 0 & 0 
        
\end{pmatrix}, \theta ( \frac{\partial r}{\partial a})=\begin{pmatrix}  
 1 & 1 \\         
 1 & 1
        
\end{pmatrix},$$ thus $H_1(ker \theta)=\Z$.

Case 4: $n=2p-1, k=2l-1$. Consider $\theta:\Gamma \to S_2$, $\theta(a)=\begin{pmatrix}  
 1 & 0 \\         
 0 & 1 
        
\end{pmatrix}$, $\theta(t)=\begin{pmatrix}  
 0 & 1 \\         
 1 & 0 
        
\end{pmatrix}$, we have $$\theta ( \frac{\partial r}{\partial t})=\begin{pmatrix}  
 0 & 0 \\         
 0 & 0 
        
\end{pmatrix}, \theta ( \frac{\partial r}{\partial a})=\begin{pmatrix}  
 -(2p-1) & 2p+1 \\         
 2p+1 & -(2p-1)
 \end{pmatrix},$$ thus $H_1(ker \theta)=\Z\oplus \Z_{8p}$.

Thus $G_{n,n+2,k}$ is virtually $2$-avoidable for all $n,k\in \Z$. Exactly similar argument proves the same for $G_{n+2,n,k}$. For $k=1$, we recover Proposition~\ref{bs}.

\subsection{Baumslag-Gersten Groups} These are one relator group of the form $$BG(m,n)=\langle a,t | a^ta^ma^{-t}=a^{n}\rangle=\langle a,t | tat^{-1}a^mta^{-1}t^{-1}=a^{n}\rangle.$$ Clearly $BG(m,n)$ are not 2-avoidable when $|m-n|=2$, hence we consider as in case of $B(m,n)$, $m=n+2$. Firstly, for $r=tat^{-1}a^mta^{-1}t^{-1}a^{-n}$, 
$$\frac{\partial r}{\partial a}=t+tat^{-1}(1+a+...+a^{m-1})-a^nt-(1+a+...+a^{n-1}),$$ $$ \frac{\partial r}{\partial t}=(1-a^n)-tat^{-1}(1-a^m).$$

Case 1: $n=2p$; Consider $\theta:\Gamma \to S_2$, $\theta(a)=\theta(t)=\begin{pmatrix}  
 0 & 1 \\         
 1 & 0 
        
\end{pmatrix}.$ We have $$\theta ( \frac{\partial r}{\partial t})=\begin{pmatrix}  
 0 & 0 \\         
 0 & 0 
        
\end{pmatrix}, \theta ( \frac{\partial r}{\partial a})=\begin{pmatrix}  
 1 & 1 \\         
 1 & 1
        
\end{pmatrix},$$ thus $H_1(ker \theta)=\Z$.

Case 2: $n=2p-1$; Consider $\theta:\Gamma \to S_2$, $\theta(a)=\theta(t)=\begin{pmatrix}  
 0 & 1 \\         
 1 & 0 
        
\end{pmatrix}.$ We have $$\theta ( \frac{\partial r}{\partial t})=\begin{pmatrix}  
 2 & -2 \\         
 -2 & 2 
        
\end{pmatrix}, \theta ( \frac{\partial r}{\partial a})=\begin{pmatrix}  
 -1 & 3 \\         
 3 & -1
        
\end{pmatrix},$$ thus $H_1(ker \theta)=\Z\oplus \Z_4$.

 \subsection{Meskin Groups} We consider the following class of one-relator groups which were considered by Stephen Meskin in~\cite{Mes}, $$\Gamma=\langle s_1,s_2,...,s_n| s_1^{k_1}s_2^{k_2}...s_n^{k_n}\rangle,$$
 where $n\ge 2$, $k_i\ge 1$. This class includes non-orientable surface groups, torus knot groups. $\Gamma$ is non-$2$-avoidable groups if and only if $gcd(k_1,...,k_n)=2$. Clearly 
 
 $$\frac {\partial r}{\partial s_i}=s_1^{k_1}...s_{i-1}^{k_{i-1}}(1+s_i+s_i^2+...+s_i^{k_i-1}),$$ for all $1\le i \le n$.  

 Consider $\Gamma'=ker \theta$, $$\theta: \Gamma \to S_2, \theta(s_i)=\begin{pmatrix}  
 0 & 1 \\         
 1 & 0 
        
\end{pmatrix} .$$ Since $k_i$'s are all even, with some calculation we have $$\theta (\frac {\partial r}{\partial s_i})=\begin{pmatrix}  
 k_i/2 & k_i/2 \\         
 k_i/2 & k_i/2 
        
\end{pmatrix}.$$ Thus clearly $$\theta (J)=\begin{pmatrix}  
 k_1/2 & k_1/2 \\         
 k_1/2 & k_1/2 \\
  k_2/2 & k_2/2 \\
  ...\\
  .\\
  .\\
   k_n/2 & k_n/2 \\
    k_n/2 & k_n/2

\end{pmatrix}.$$ Since $gcd(k_1/2,k_2/2,...,k_n/2)=1$, the above matrix clearly has Smith Normal form having just two entries $1$, thus $H_1(\Gamma')=\Z^2$.

\section{Acknowledgement}
The author thanks his advisor Alexander Dranishnikov for extremely helpful discussions his support. The author also thanks Lorenzo Ruffoni for helpful conversation regarding various topics about the project.

\end{document}